\documentclass[12pt,
a4paper]{article}

\usepackage{CJK}
\usepackage{CJKspace}
\usepackage{hyperref}
\usepackage{amsfonts}
\usepackage{mathrsfs}
\usepackage{comment}
\usepackage{amsmath}
\usepackage{amssymb}
\usepackage{amsthm}
\usepackage{amscd}
\usepackage{graphicx}
\usepackage{indentfirst}
\usepackage[all]{xy}
\usepackage{titlesec}
\usepackage[top=25.4mm, bottom=25.4mm, left=31.7mm, right=32.2mm]{geometry}
\usepackage{enumerate}
\usepackage{bm}
\usepackage{enumitem}
\newtheorem{theorem}{Theorem}[section]
\newtheorem{lemma}[theorem]{Lemma}
\newtheorem{proposition}[theorem]{Proposition}
\newtheorem{corollary}[theorem]{Corollary}
\newtheorem{conjecture}[theorem]{Conjecture}
\newtheorem{definition}[theorem]{Definition}

\newtheorem{remark}[theorem]{Remark}

\newtheorem{example}[theorem]{Example}

\newcommand{\ma}{\mathcal}

\newcommand{\s}{\subseteq}

\newcommand{\fr}{\frac}
\newcommand{\lc}{\lceil}
\newcommand{\rc}{\rceil}

\begin{document}
\title{Some intriguing upper bounds for separating hash families}

\author{Gennian Ge$^{\text{a,}}$\thanks{Corresponding author. Email address: gnge@zju.edu.cn. Research supported by the National Natural Science Foundation of China under Grant Nos. 11431003 and 61571310, Beijing Scholars Program,  Beijing Hundreds of Leading Talents Training Project of Science and Technology, and Beijing Municipal Natural Science Foundation.},
Chong Shangguan$^{\text{b,}}$\thanks{Email address: theoreming@163.com.},
Xin Wang$^{\text{c,}}$\thanks{Email address: xinw@suda.edu.cn. Research supported in part by the Post-Doctoral Science Foundation of China under Grant 2018M632356 and in part by the National Natural Science Foundation of China under Grant No. 11801392.}\\
\footnotesize $^{\text{a}}$ School of Mathematical Sciences, Capital Normal University, Beijing 100048, China.\\
\footnotesize $^{\text{b}}$ School of Mathematical Sciences, Zhejiang University, Hangzhou 310027, Zhejiang, China.\\
\footnotesize $^{\text{c}}$ Department of Mathematics, Soochow University, Suzhou 215006, China.\\
}
\date{}
\maketitle

\begin{abstract}

    An $N\times n$ matrix on $q$ symbols is called $\{w_1,\ldots,w_t\}$-separating if for arbitrary $t$ pairwise disjoint column sets $C_1,\ldots,C_t$ with $|C_i|=w_i$ for $1\le i\le t$,
    there exists a row $f$ such that $f(C_1),\ldots,f(C_t)$ are also pairwise disjoint, where $f(C_i)$ denotes the collection of components of $C_i$ restricted to row $f$.
    Given integers $N,q$ and $w_1,\ldots,w_t$, denote by $C(N,q,\{w_1,\ldots,w_t\})$ the maximal $n$ such that a corresponding matrix does exist.
    The determination of $C(N,q,\{w_1,\ldots,w_t\})$ has received remarkable attentions during the recent years.
    The main purpose of this paper is to introduce two novel methodologies to attack the upper bound of $C(N,q,\{w_1,\ldots,w_t\})$.
    The first one is a combination of the famous graph removal lemma in extremal graph theory and a Johnson-type recursive inequality in coding theory, and the second one is the probabilistic method.
    As a consequence, we obtain several intriguing upper bounds for some parameters of $C(N,q,\{w_1,\ldots,w_t\})$, which significantly improve the previously known results.
\end{abstract}



\section{Introduction}

Separating hash families are useful combinatorial objects introduced by Stinson, Wei and Chen \cite{ST08} in 2008.
They are generalizations of many previously known and well-studied combinatorial objects.
Therefore, during the recent years they have received remarkable attentions.
Many efforts have been made to determine the upper and lower bounds for separating hash families, see for example, \cite{Trung2011}, \cite{blackburn08}, \cite{shf} and\cite{ST08}.

Let $X$ and $Y$ be sets of cardinalities $n$ and $q$, respectively.
We call a family $\ma{F}$ of $N$ functions $f:X\rightarrow Y$ an {\it $(N;n,q)$-hash family}.
Let $f:X\rightarrow Y$ be a function and consider pairwise disjoint subsets $C_1,\ldots,C_t\s X$.
We say that $f$ {\it separates} $C_1,\ldots,C_t$ if $f(C_1),\ldots,f(C_t)$ are pairwise disjoint subsets of $Y$.
We further say that $\ma{F}$ is an {\it $(N;n,q,\{w_1,\ldots,w_t\})$-separating hash family} (which will be also denoted as an $SHF(N;n,q,\{w_1,\ldots,w_t\})$) if it satisfies the following property:
For all pairwise disjoint subsets $C_1,\ldots,C_t\s X$ with $|C_i|=w_i$ for $1\le i\le t$, there exists at least one function $f\in\ma{F}$ that separates $C_1,\ldots,C_t$.
We call the multiset $\{w_1,\ldots,w_t\}$ the {\it type} of this separating hash family and denote $u:=\sum_{i=1}^t w_i$ throughout this paper.
Without loss of generality, we may fix the alphabet set $Y$ to be $[q]$, where $[q]:=\{1,\ldots,q\}$.
We also assume that $w_1\le\cdots\le w_t$.

When $w_1=\cdots=w_t=1$, an $SHF(N;n,q,\{1,\ldots,1\})$ is also known to be a {\it $t$-perfect hash family}. Perfect hash families were introduced by Mehlhorn \cite{data1} in 1984 and have found applications in cryptography \cite{perfect3}, \cite{N=3t=3}, database management \cite{data1} and the designs of circuits \cite{circuit} and algorithms \cite{alog}.
When $t=2$, $w_1=1$ and $w_2=w\ge 2$, an $SHF(N;n,q,\{1,w\})$ is known as a {\it $w$-frameproof code}.
When $t=2$, $w_1=w_2=w\ge 2$, an $SHF(N;n,q,\{w,w\})$ is known as a {\it $w$-secure-frameproof code}.
Frameproof and secure-frameproof codes were introduced by Chor, Fiat and Noar \cite{chor1994} in 1994, and by Stinson, Trung and Wei \cite{secure} in 2000, respectively.
They can serve as techniques to prevent copyrighted materials from unauthorized use \cite{BL03}, \cite{ST01}.
Moreover, codes with the {\it identifiable parent property} \cite{ipp}, \cite{shangguanipp} are separating hash families which are simultaneously of type $\{1,1,1\}$ and $\{2,2\}$, and {\it partially hash families} \cite{ippb} are separating hash families satisfying $w_1=\cdots=w_{t-1}=1$ and $w_t=w\ge 2$. Both of them have applications in privacy protection.

The determinations of upper and lower bounds for separating hash families are important open problems in this research area.
Given integers $N,q$ and $w_1,\ldots,w_t$, denote by $C(N,q,\{w_1,\ldots,w_t\})$ the maximal $n$ such that there exists an $SHF(N;n,q,\{w_1,\ldots,w_t\})$.
In the literature, there are two major known approaches to attack the bound $C(N,q,\{w_1,\ldots,w_t\})$.
The first one is called the grouping coordinates method, which states that $C(aN,q,\{w_1,\ldots,w_t\})\le C(N,q^{a},\{w_1,\ldots,w_t\})$ holds for every positive integer $a$.
This inequality can be proved directly by regarding a family of $q$-ary length $aN$ vectors as a family of $q^a$-ary length $N$ vectors.
One can verify that the new family preserves the $\{w_1,\ldots,w_t\}$-separating property as long as the original family does.
Due to this inequality, the problem of bounding $C(N,q,\{w_1,\ldots,w_t\})$ can be reduced to bounding $C(u-1,q,\{w_1,\ldots,w_t\})$, since one can show
$C(N,q,\{w_1,\ldots,w_t\})\le C(u-1,q^{\lceil N/(u-1)\rceil},\{w_1,\ldots,w_t\})$.
In 2008, Blackburn et al. \cite{blackburn08} obtained a linear bound (linear in $q$) for $C(u-1,q,\{w_1,\ldots,w_t\})$, which states that $C(u-1,q,\{w_1,\ldots,w_t\})\le (w_1w_2+u-w_1-w_2)q$.
This bound immediately implies the general upper bound $C(N,q,\{w_1,\ldots,w_t\})\le(w_1w_2+u-w_1-w_2)q^{\lc\fr{N}{u-1}\rc}$.
An improved bound was obtained by Bazrafshan and Trung \cite{Trung2011} in 2011, showing that $C(u-1,q,\{w_1,\ldots,w_t\})\le (u-1)q$.
Moreover, they conjectured that $u-1$ was the smallest linear factor satisfying this bound.
One can easily see that the upper bound derived from the grouping coordinates method could never be better than $C(N,q,\{w_1,\ldots,w_t\})\le(u-1)q^{\lc\fr{N}{u-1}\rc}$,
provided that Bazrafshan and Trung's conjecture was true.
Indeed, recently the correctness of this conjecture was confirmed by the first two authors of this paper in \cite{shf}.
Moreover, in the same paper the authors have introduced a new method to study $C(N,q,\{w_1,\ldots,w_t\})$.
It was shown that $C(N,q,\{w_1,\ldots,w_t\})$ satisfies the following Johnson-type recursive inequality.

\begin{lemma} [Johnson-type bound]\label{johnson}
    Let $1\le l\le N$ be a positive integer, then it holds that $C(N,q,\{w_1,...,w_t\})\le q^l+\max\{u-1,C(N-l,q,\{w_1-1,...,w_t\})\}$.
    Indeed, in the right side of the inequality we can choose the minus of 1 to be after arbitrary $w_i,~1\le i\le t$.
\end{lemma}

Using this inequality, the best known general upper bound for $C(N,q,\{w_1,\ldots,w_t\})$ was derived in \cite{shf}, which is stated as the following theorem.

\begin{theorem} \label{recusivebd}
   Suppose there exists an $SHF(N;n,q,\{w_1,\ldots,w_t\})$. Let $u=\sum_{i=1}^t w_i$ and let $1\le r\le u-1$ be the positive integer such that $N\equiv r \pmod{u-1}$.
   If $C(\lfloor N/(u-1)\rfloor,q,\{w_1,\ldots,w_t\})\ge u$, then it holds that $n\le rq^{\lceil N/(u-1)\rceil}+(u-1-r)q^{\lfloor N/(u-1)\rfloor}$.
\end{theorem}

Obviously, for $(u-1)\nmid N$ the above bound is an improvement of the previous ones since $r<u-1\le w_1w_2+u-w_1-w_2$. Indeed it behaves very well for fixed $w_1,\ldots,w_t,N$ and sufficiently large $q$.
The exponent $\lc\fr{N}{u-1}\rc$ is realistic in the sense that there are probabilistic constructions (see \cite{perfect2} for a proof using the Lov$\acute{a}$sz Local Lemma or Theorem 2.1 of \cite{ippc} for a proof using the alteration method) showing that
\begin{equation*}
  \begin{aligned}
    C(N,q,\{w_1,\ldots,w_t\})\ge\fr{1}{2^u}(\fr{1}{1-g(q,u)})^{\fr{N}{u-1}},
  \end{aligned}
\end{equation*}

\noindent where $g(q,u)=\fr{q(q-1)\cdots(q-u+1)}{q^u}$, which implies that

\begin{equation*}
  \begin{aligned}
    C(N,q,\{w_1,\ldots,w_t\})>\fr{1}{2^u}(\fr{q}{\binom{u}{2}})^{\fr{N}{u-1}}=\Omega_{u,N}(q^{\fr{N}{u-1}})
  \end{aligned}
\end{equation*}

\noindent holds for sufficiently large $q$ and fixed $w_1,\ldots,w_t,N$.
It follows that the exponent $\lc\fr{N}{u-1}\rc$ cannot be further reduced when $(u-1)\mid N$. It is an open problem \cite{blackburn08} to determine whether the exponent is tight when $(u-1)\nmid N$.
After many efforts \cite{fuji2015perfect}, \cite{N=3t=3}, the first breakthrough results in this direction were obtained in \cite{shf}.
Answering an open problem of Walker II and Colbourn \cite{N=3t=3}, it was shown
\begin{equation}\label{111}
  \begin{aligned}
    q^{2-o(1)}<C(3,q,\{1,1,1\})=o(q^2)
  \end{aligned}
\end{equation}
and
\begin{equation}\label{1111}
  \begin{aligned}
    q^{2-o(1)}<C(4,q,\{1,1,1,1\})=o(q^2)
  \end{aligned}
\end{equation}

\noindent hold for sufficiently large $q$.
In this paper we will develop a systematic method to study the upper bound for $C(N,q,\{w_1,\ldots,w_t\})$ when $(u-1)\nmid N$.
We will extend the upper bounds of (\ref{111}) and (\ref{1111}) to more general parameters.

On the other hand, the bound presented in Theorem \ref{recusivebd} does not behave quite well for relatively small $q$ (compared with $u$ and $N$).
For example, in \cite{Korner} and \cite{nilli} it was shown that for $w_1=\cdots=w_t=1$ it holds that
\begin{equation}\label{phfbound}
  \begin{aligned}
    C(N,q,\{1,\ldots,1\})\le\min_{0\le j\le t-2}(t-j-1)(\fr{q-j}{t-j-1})^{(1+o(1))g(q,j+1)N},
  \end{aligned}
\end{equation}
where the $o(1)$ term is induced by the estimation $\binom{n}{t}\thickapprox\fr{n^t}{t!}$. Setting $q=t$ one can deduce from (\ref{phfbound}) that
\begin{equation}\label{phfbound2}
  \begin{aligned}
    C(N,t,\{1,\ldots,1\})
    \le\min\{2^{(1+o(1))\fr{t!}{t^{t-1}}N},(t-1)(\fr{t}{t-1})^N\},
  \end{aligned}
\end{equation}
which is much better than the bound $C(N,t,\{1,\ldots,1\})=\ma{O}(t^{\lc\fr{N}{t-1}\rc})$ obtained by Theorem \ref{recusivebd}.
The paper \cite{separatingcodes} contains new upper bounds on $SHF(N,n,q,\{w_1,w_2\})$ for some small $q$ and small $w_1,w_2$.
In \cite{22separating} it was shown that $C(N,2,\{2,2\})\le\ma{O}(2^{0.28N})$ which is superior to the bound $\ma{O}(2^{0.33N})$ obtained by Theorem \ref{recusivebd}.

For the very special case $t=2$ and $q=2$, we use $N(w)$ to denote the minimal integer $N$ such that there exists an $SHF(N;n,2,\{1,w\})$ satisfying $n>N$. 
The determination of $N(w)$ has received considerable attentions. In \cite{binaryfpc1} it was shown $N(w)\ge 3w$.
The best known result for $N(w)$ was obtained by Miao and the authors of this paper in \cite{shangguan2014new}, which shows that
\begin{equation}\label{N(w)shangguan}
  \begin{aligned}
    N(w)>\binom{w+1}{2}.
  \end{aligned}
\end{equation}

The main purpose of this paper is to introduce two novel methods, namely, the graph removal lemma (see Lemma \ref{graphremoval} below) combined with the Johnson-type inequality,
and the probabilistic method, to attack the upper bound of $C(N,q,\{w_1,\ldots,w_t\})$.
As a result, we can improve the bound presented in Theorem \ref{recusivebd} in various aspects.
\begin{itemize}
\item Firstly, it was widely believed that $C(4,q,\{2,2\})\le c_1q^2+c_2q$ and many efforts had been made to improve the constants $c_1$ and $c_2$, see for example, \cite{trung2014}, \cite{ST01}, \cite{secure}, \cite{ST08} and \cite{secureshf}. Recently, the best upper bound along this line is proved by Niu and Cao in \cite{Niu2016Some}, which states that $C(2w;q,\{w,w\})\le(q-1)^2+1$.
However, for sufficiently large $q$, using the graph removal lemma we can prove an asymptotically optimal bound showing that $q^{2-o(1)}<C(4,q,\{2,2\})=o(q^2)$.
Our bound is somewhat surprising in the sense that it is the first bound in the literature showing that the order $q^2$ is not achievable.
By combining the graph removal lemma and the Johnson-type recursive inequality, we are able to prove a much more general bound (see Theorem \ref{upperbound} below), showing that $C(w_1+\cdots+w_t,q,\{w_1,\ldots,w_t\})=o(q^2)$ holds for arbitrary $w_1,\ldots,w_t$ except for $t=2$ and $\{w_1,w_2\}=\{1,w\}$.

\item Secondly, for the special case $w_1=\cdots=w_t=w$ and $q=t$, previously known result implies that $C(N,t,\{w,\dots,w\})=\ma{O}(t^{\lc\fr{N}{tw-1}\rc})$.
By employing a probabilistic argument, we can show $C(N,t,\{w,\dots,w\})=t^{\ma{O}(\fr{(t!)^2N}{t^{tw-1}})}$ (see Theorem \ref{lowerbd} below),
which substantially improves the bound obtained by Theorem \ref{recusivebd}.

\item Thirdly, by relating the determination of $C(N,2,\{1,w\})$ to an old conjecture of Erd\H{o}s, Frankl and F{\"u}redi on cover-free families \cite{CFF}, we improve the result of (\ref{N(w)shangguan}) by showing that $N(w)>\frac{15+\sqrt{33}}{24} (w-2)^2$ (see Theorem \ref{shangguanbinary} below), which also implies a tight bound that $C(N,2,\{1,w\})=N$ for $N\le\frac{15+\sqrt{33}}{24} (w-2)^2$.
\end{itemize}

The rest of this paper is organized as follows. Section 2 contains the necessary notations and definitions for the proofs in this paper.
In Section 3 we will use the graph removal lemma to derive an upper bound for a problem in graph theory.
The bound itself is of independent interest and can also be used to derive an upper bound for $SHF(4;n,q\{2,2\})$.
In Section 4, we will use the Johnson-type recursive inequality to prove our general bound for $SHF(\sum_{i=1}^t w_i;n,q,\{w_1,\ldots,w_t\})$.
In Section 5, we will use a probabilistic method to study the behavior of $C(N,q,\{w_1,\ldots,w_t\})$ for the special case $q=t$.
In Section 6, we will relate the determination of $N(w)$ to a conjecture of Erd\H{o}s, Frankl and F{\"u}redi and present an improved upper bound for $N(w)$.
Section 7 contains some concluding remarks and open problems.

\section{Preliminaries}

    \subsection{Representation matrices, separating hash families and hypergraphs}

    In general, one can regard a separating hash family as a matrix.
   An $N\times n$ matrix $M$ on $q$ symbols is called $\{w_1,\ldots,w_t\}$-separating if for arbitrary $t$ pairwise disjoint column sets $C_1,\ldots,C_t$ with $|C_i|=w_i$ for $1\le i\le t$,
   there exists a row $f$ such that $f(C_1),\ldots,f(C_t)$ are also pairwise disjoint, where $f(C_i)$ denotes the collection of components of $C_i$ restricted to row $f$.
   We call $M$ the {\it representation matrix} of an $SHF(N;n,q,\{w_1,\ldots,w_t\})$.

   We will use hypergraphs to study separating hash families. 
   A hypergraph $\ma{H}=(V(\ma{H}),E(\ma{H}))$ can be viewed as a pair of vertices and edges,
   where the vertex set $V(\ma{H})$ can be regarded as a finite set $X$ and the edge set $E(\ma{H})$ can be regarded as a collection of subsets of $X$.
   For the sake of simplicity, we write $\ma{H}$ to represent the edge set $E(\ma{H})$, and hence $|\ma{H}|$ stands for $|E(\ma{H})|$.
  A hypergraph $\ma{H}$ is said to be linear if for all distinct $A,B\in \ma{H}$ it holds that $|A \cap B|\le1$. Furthermore, we say $\ma{H}$ is $r$-uniform if $|A|=r$ for every $A\in \ma{H}$.

  An $r$-uniform hypergraph $\ma{H}$ is $r$-partite if its vertex set $V(\ma{H})$ can be colored in $r$ colors in such a way that no edge of $\ma{H}$ contains two vertices of the same color.
  In such a coloring, the color classes of $V(\ma{H})$, i.e., the sets of all vertices of the same color, are called vertex parts of $\ma{H}$.
  We use $V_1,\ldots,V_r$ to denote the $r$ color classes of $V(\ma{H})$. Then $V(\ma{H})$ is a disjoint union of the $V_i$'s and for every $A\in\ma{H}$, $|A\cap V_i|=1$ holds for each $1\le i\le r$.

  One can also regard an $N$-uniform $N$-partite hypergraph $\ma{H}$ with $|V_1|=\cdots=|V_N|=q$ as an $N\times|\ma{H}|$ $q$-ary matrix $M$, which can be realized as follows.
  The rows of $M$ are indexed by the vertex parts of $\ma{H}$ and columns of $M$ are indexed by the edges of $\ma{H}$.
  For each $1\le i\le r$, without loss of generality we can assume that $V_i=\{v_{i1},\ldots,v_{iq}\}$, then $V(\ma{H})=\{v_{ij}\mid 1\le i\le N,~1\le j\le q\}$ is well-defined.
  Every edge $A\in\ma{H}$ (i.e. column of $M$) must have the form $A=\{v_{1j_1},\ldots,v_{Nj_N}\}$, where $1\le j_k\le q$ for each $1\le k\le N$.
  The entry in row $i$ and column $A$ is just the vertex $V_i\cap A\in V(\ma{H})$.
  We call $M$ the {\it representation matrix} of the hypergraph $\ma{H}$.
  Observe that $|V(\ma{H})|=\sum_{i=1}^N|V_i|=Nq$, it follows that $M$ is an $Nq$-ary matrix.
  However, when investigating the separating property of such a matrix, we only want to know whether there exists a row that can separate several sets of columns.
  Thus two entries of $M$ located in two different rows will not interact each other.
  So for simplicity one can just view $M$ as a $q$-ary matrix by setting $V_i=\{v_{i1},\ldots,v_{iq}\}=\{1,\ldots,q\}$ for each $1\le i\le N$.

    \subsection{Hypergraph rainbow cycles}

    To study the upper bound of $C(4,q,\{2,2\})$ we will use the notion of rainbow 4-cycles. The formal definition of rainbow cycles is presented as follows.
    \begin{definition}
      Let $\ma{H}$ be an $r$-uniform $r$-partite linear hypergraph with vertex parts $V_1,\ldots,V_r$. A rainbow $k$-cycle of $\ma{H}$ is an alternating sequence of vertices and edges of the form $v_1,E_1,v_2,E_2,\ldots,v_k,E_k,v_1$ such that
       \begin{itemize}
      \item [(a)]$v_1,v_2,\ldots,v_k$ are distinct vertices of $\ma{H}$,
      \item [(b)]$E_1,E_2,\ldots,E_k$ are distinct edges of $\ma{H}$,
      \item [(c)]$v_i,v_{i+1}\in E_i$ for $1\le i\le k-1$ and $v_k,v_1\in E_k$,
      \item [(d)]For $1\le i\le k$, $v_i\in V_{j_i}$ and $j_i\neq j_{i'}$ for arbitrary $i\neq i'$.
    \end{itemize}
    \end{definition}
    \noindent Condition $(d)$ implies that all $k$ vertices are located in $k$ different vertex parts.
    It is easy to see that for $r$-partite hypergraphs, a rainbow $k$-cycle exists only if $k\le r$.
    In Section 3 we will present an upper bound for the maximal number of edges of a linear hypergraph that contains no rainbow cycles (see Theorem \ref{rainbowcycle} below).
    In Section 4 this bound will be used to derive an upper bound for separating hash families.

\subsection{Cover-free families and a conjecture of Erd\H{o}s, Frankl and F{\"u}redi}

Let $X$ be a set of $N$ elements. A family $\ma{F}\s 2^{X}$ is said to be $w$-cover-free if for arbitrary distinct $w+1$ members $A_0,A_1,\ldots,A_w$ of $\ma{F}$
it holds that $A_0\nsubseteq A_1\cup A_2\cup\cdots\cup A_w$.
Suppose that $|\ma{F}|=n$ and let us denote $X=\{x_1,\ldots,x_N\}$ and $\ma{F}=\{A_1,\ldots,A_n\}$.
$\ma{F}$ will be denoted as a $CFF(N;n,w)$.
Denote by $M^*$ the representation matrix of $\ma{F}$, which is an $N\times n$ binary matrix whose rows are indexed by the elements of $X$ and whose columns are indexed by the members of $\ma{F}$, such that the entry in the $i$-th row and the $j$-th column is 1 if and only if $x_i\in A_j$.
In a binary matrix, the weight of a column is simply the number of 1's contained in it. We will need a property of cover-free families.

\begin{lemma}[\cite{cffpro}]\label{cffpro}
    Let $\ma{F}$ be a $CFF(N;n,w)$ with representation matrix $M^*$. Fix an arbitrary member $A$ of $\ma{F}$ and consider the new family $\ma{F}_1$ defined by
    \begin{enumerate}
      \item [1)] $\ma{F}_1\s 2^{X\setminus A}$,
      \item [2)] $\ma{F}_1=\{B\setminus A:B\in\ma{F},~B\neq A\}$.
    \end{enumerate}
    \noindent Then $\ma{F}_1$ is a $CFF(N-|A|;n-1,w-1)$.
\end{lemma}

\begin{proof}
    The first two parameters in $CFF(N-|A|;n-1,w-1)$ are easy to verify. It suffices to prove that $\ma{F}_1$ is a $(w-1)$-cover-free family. Suppose otherwise, there are $w$ different members $B_0,B_1,\ldots,B_{w-1}$ of $\ma{F}_1$ such that $B_0\s B_1\cup\cdots \cup B_{w-1}$. For each $0\le i\le w-1$, denote $A_i$ the member of $\ma{F}$ such that $B_i=A_i\setminus A$. Then it holds that $A_0\s B_0\cup A\s (B_1\cup\cdots B_{w-1})\cup A\s A_1\cup\ldots\cup A_{w-1}\cup A$, which violates the $w$-cover-free property of $\ma{F}$.
\end{proof}

Denote $N^*(w)$ the minimal $N$ such that there exists a $CFF(N;n,w)$ with $n>N$. In 1985 Erd\H{o}s, Frankl and F{\"u}redi \cite{CFF} posed the following conjecture:

\begin{conjecture}[\cite{CFF}]\label{cff}
  $\lim_{w\rightarrow\infty}\fr{N^*(w)}{w^2}=1$, or in an even stronger form $N^*(w)\ge(w+1)^2$.
\end{conjecture}

Note that when $w+1$ is a prime power, an affine plane of order $w+1$ induces an example for $CFF((w+1)^2;(w+1)^2+(w+1),w)$.
The best known lower bound for $N^*(w)$ is proved in \cite{shangguan2015gptesting}, which is restated as follows.

\begin{lemma}[\cite{shangguan2015gptesting}]\label{grouptesting}
    $N^*(w)\ge \fr{15+\sqrt{33}}{24} w^2$.
\end{lemma}

\noindent By establishing a bridge between cover-free families and separating hash families, this lemma will be used to prove a lower bound for $N(w)$ in Section 6.

\section{The maximal number of edges of a linear hypergraph which contains no rainbow cycles}

Rainbow cycles are closely related to separating hash families in the sense that an $SHF(2w;n,q,\{w,w\})$ cannot contain a rainbow $2w$-cycle.
In the following we will present two examples to illustrate this observation.

\begin{example}\label{eg4cycle}
   Let $\ma{F}$ be an $SHF(4;n,q,\{2,2\})$ and $\ma{H}_{\ma{F}}$ be the hypergraph defined by the representation matrix of $\ma{F}$.
   Assume $\ma{H}_{\ma{F}}$ contains a rainbow 4-cycle $$v_4,A_1,v_1,A_2,v_2,A_3,v_3,A_4,v_4,$$ which can be depicted as Table \ref{rain4-cycle}.
   Obviously, no row of the representation matrix can separate $\{A_1,A_3\}$ and $\{A_2,A_4\}$, which violates the $\{2,2\}$-separating property.

    \begin{table}[ht]
  \begin{center}
    \begin{tabular}{|c|c|c|c|c|}
      \hline
       & $A_1$ & $A_2$ & $A_3$ & $A_4$ \\\hline
      $V_1$ & $v_1$ & $v_1$ & &\\\hline
      $V_2$ &  & $v_2$ & $v_2$ & \\\hline
      $V_3$ &  &  & $v_3$ & $v_3$\\\hline
      $V_4$ & $v_4$ &  &  & $v_4$\\\hline
    \end{tabular}
  \end{center}
   \caption{Rainbow 4-cycle}\label{rain4-cycle}
    \end{table}
\end{example}

\begin{example}\label{rainbow2w-cycle}
Let $\ma{F}$ be an $SHF(2w;n,q,\{w,w\})$ and $\ma{H}_{\ma{F}}$ be the hypergraph defined by the representation matrix of $\ma{F}$.
Assume $\ma{H}_{\ma{F}}$ contains a $2w$-cycle, which can be denoted as $$v_{2w},A_1,v_1,A_2,v_2,A_3,\ldots,A_{2w-1},v_{2w-1},A_{2w},v_{2w}.$$
As in Example \ref{eg4cycle}, this $2w$-cycle can be depicted as Table \ref{rain2w-cycle}. Obviously, no row of the representation matrix can separate $\{A_1,A_3,\ldots,A_{2w-1}\}$ and $\{A_2,A_4,\ldots,A_{2w}\}$, which violates the $\{w,w\}$-separating property.

      \begin{table}[h]
  \begin{center}
\begin{tabular}{|c|c|c|c|c|c|c|c|}
  \hline
   & $A_1$ & $A_2$ & $A_3$ & $\cdots$ & $\cdots$  & $A_{2w-1}$ & $A_{2w}$ \\\hline
  $V_1$ & $v_1$ & $v_1$ &  &  &  &    &\\\hline
  $V_1$ &  & $v_2$ & $v_2$ &  &  &    &\\\hline
  $\vdots$ &  &   & $\ddots$  & $\ddots$   &  &  &\\\hline
  $\vdots$ &   &   &   &$\ddots$ & $\ddots$   &  &\\\hline
  $\vdots$ &   &   &   & & $\ddots$ & $\ddots$   &\\\hline
  $V_{2w-1}$ &  &  &  &  &  & $v_{2w-1}$ & $v_{2w-1}$ \\\hline
  $V_{2w}$ & $v_{2w}$ &  &  &  &  &  & $v_{2w}$ \\
  \hline
\end{tabular}
  \end{center}
   \caption{Rainbow $2w$-cycle}\label{rain2w-cycle}
    \end{table}
\end{example}

\begin{lemma}[Graph removal lemma, see for example \cite{graphremovallemmas}]\label{graphremoval}
  For any graph $G$ and any $\epsilon>0$, there exists $\delta>0$ such that
  any graph on $n$ vertices which contains at most $\delta n^{v(G)}$ copies of $G$ can be made $G$-free by removing at most $\epsilon n^2$ edges.
\end{lemma}

Using the graph removal lemma described in Lemma \ref{graphremoval}, it is straightforward to deduce the following fact: For any given constant $\epsilon>0$, there exists some $\delta(\epsilon)>0$ such that if one must delete at least $\epsilon n^2$ edges to make a graph $H$ with $n$ vertices $G$-free, then $H$ must contain at least $\delta(\epsilon)n^{v(G)}$ copies of $G$.

\begin{theorem}\label{rainbowcycle}
Let $n$ be a sufficiently large positive integer and $\epsilon>0$ be a given positive constant. Let $r$ be a fixed positive integer.
Assume $\ma{H}$ is an $r$-uniform $r$-partite linear hypergraph with $\epsilon n^2$ edges.
Then $\ma{H}$ must contain a rainbow $k$-cycle for every integer $3\le k\le r$.
\end{theorem}

\begin{proof}
  Note that a complete graph is an ordinary graph whose every two distinct vertices form an edge.
  We use $K_n$ to denote a complete graph on $n$ vertices and use $K_n(X)$ to emphasize the vertex set $X$.
  Let us form an auxiliary graph $\ma{H}^*$ as follows.
   For any $r$-edge $A\in\ma{H}$, consider the complete graph $K_r(A)$ which is defined on the vertex set $A$.
  The new graph $\ma{H}^*$ is an ordinary graph whose edge set is formed by taking together all of the edges of $K_r(A)$; in other words, $\ma{H}^*=\cup_{A\in\ma{H}}K_r(A)$.
  Observe that for any distinct $A,B\in\ma{H}$, it holds that $|A\cap B|\le 1$.
  This simple observation reflects an important property of $\ma{H}^*$, namely, $\ma{H}^*$ contains $|\ma{H}|~(=\epsilon n^2)$ edge-disjoint copies of $K_r$.
  Thus in order to make $\ma{H}^*$ $K_r$-free, at least $\epsilon n^2$ edges of it must be deleted.
  By the graph removal lemma one can infer that $\ma{H}^*$ contains at least $\delta(\epsilon)n^r$ copies of $K_r$,
  where $\delta(\epsilon)$ is a positive constant guaranteed by the graph removal lemma and it depends only on $\epsilon$.

  In what follows we will argue that $\ma{H}^*$ must contain a $K_r^*$ satisfying the property that all the edges of it must be induced by distinct $r$-edges of $\ma{H}$.
  This statement can be proved by the following straightforward counting argument.
  Let us count the number of copies of $K_r$'s in $\ma{H}^*$ which contains at least two edges arising from a same $r$-edge $A\in\ma{H}$.
  First of all, notice that two edges of $K_r$ determine at least $3$ of its vertices, and there are at most $|\ma{H}|~(=\epsilon n^2)$ choices for such an $A\in\ma{H}$.
  Moreover, it is obvious that the remaining $r-3$ vertices of the undetermined $K_r$ can have at most $n^{r-3}$ choices.
  Therefore, the number of such $K_r$'s is at most $\binom{r}{3}\cdot\epsilon n^2\cdot n^{r-3}=\ma{O}(n^{r-1})$, which will be strictly less than $\delta(\epsilon)n^r$ when $n$ is sufficiently large. 

Therefore, we can conclude that there always exists a $K_r^*\s\ma{H}^*$ whose $\binom{r}{2}$ edges are induced by $\binom{r}{2}$ distinct $r$-edges of $\ma{H}$.
It is easy to see that all $r$ vertices of $K_r^*$ are located in $r$ distinct vertex parts of $\ma{H}$, since any two of them are covered by an edge of $\ma{H}$ and $\ma{H}$ is $r$-partite.
For any integer $3\le k\le r$, take an arbitrary $K_k^*\s K_r^*$. This can be done since we have assumed that $k\le r$.
Label the vertices of $K_k^*$ by $\{v_1,\ldots,v_k\}$. Without loss of generality, we can assume that for $1\le i\le k-1$, $\{v_i,v_{i+1}\}\s A_{i,i+1}\in\ma{H}$ and $\{v_k,v_1\}\s A_{k,1}\in\ma{H}$. Obviously, $v_1,A_{1,2},v_2,A_{2,3},v_3\ldots,v_k,A_{k,1},v_1$ form a rainbow $k$-cycle. Therefore, the theorem follows from the obvious contradiction.
\end{proof}

\section{An upper bound for $SHF(\sum_{i=1}^t w_i;n,q,\{w_1,\ldots,w_t\})$}

The goal of this section is to prove the following general upper bound.
\begin{theorem}\label{upperbound}
  Let $\ma{F}$ be an $SHF(\sum_{i=1}^t w_i;n,q,\{w_1,\ldots,w_t\})$, where the $w_i$'s are fixed positive integers and $q$ is a sufficiently large integer. Then for either $t\ge 3$ or $t=2$ and $\min\{w_1,w_2\}\ge 2$, it holds that $n=o(q^2)$.
\end{theorem}

\begin{remark}
  The only case excluded by Theorem \ref{upperbound} is $t=2$ and $\{w_1,w_2\}=\{1,w\}$.
  We have mentioned that a $\{1,w\}$-separating hash family is equivalent to a $w$-frameproof code.
  We can use Reed-Solomon Codes (see for example, \cite{BL03}) to construct frameproof codes satisfying $C(N,n,q,\{1,w\})=\Omega(q^{\lc\fr{N}{w}\rc})$ for sufficiently large $q$ (say, $q\ge N$).
\end{remark}

We will need some lemmas before presenting the proof of Theorem \ref{upperbound}.
We say a separating hash family is {\it linear} if in its representation matrix any two distinct columns can be separated by at least $N-1$ rows, i.e., they have identical component in at most one row.
 The following corollary is a simple consequence of Theorem \ref{rainbowcycle} and Example \ref{rainbow2w-cycle}.

\begin{corollary}\label{linearshf}
  Let $\ma{F}$ be a linear $SHF(2w;n,q,\{w,w\})$. Then for fixed $w$ and sufficiently large $q$, it holds that $n=o(q^2)$.
\end{corollary}

\begin{proof}
  Note that the hypergraph $\ma{H}_{\ma{F}}$ defined by the representation matrix of $\ma{F}$ is a $2w$-uniform $2w$-partite linear hypergraph with equal part size $q$.
  If $n\ge\epsilon q^2$ for some positive constant $\epsilon>0$, then by Theorem \ref{rainbowcycle} $\ma{H}_{\ma{F}}$ must contain a rainbow $2w$-cycle.
  This can not happen according to the observation illustrated by Example \ref{rainbow2w-cycle}.
\end{proof}

Note that Corollary \ref{linearshf} only provides an upper bound for linear separating hash families.
In order to obtain an upper bound for general separating hash families, a possible strategy is to argue that every separating hash family must contain a sufficiently large linear subfamily.
In this way the upper bound for linear families can induce a general upper bound.
However, this may not be true for an arbitrary separating hash family. 
Fortunately, we find that with the help of the Johnson-type lemma, to deduce a general upper bound it suffices to show that every $SHF(4;n,q,\{2,2\})$ must contain a sufficiently large linear subfamily.

We say a column $x$ (which can be denoted as $x=(x(1),x(2),x(3),x(4))$) of an $\ma{F}:=SHF(4;n,q,\{2,2\})$ contains a {\it special coordinate} if there exists some row $i$, $1\le i\le 4$, such that there is at most one column $y\in\ma{F}\setminus\{x\}$ satisfying $x(i)=y(i)$. A column of $\ma{F}$ is said to be {\it special} if it contains a special coordinate. 
It can be easily seen that if a column $x$ is not special, then for each $1\le i\le 4$ there exist two distinct columns $y,z\in\ma{F}\setminus\{x\}$ satisfying $x(i)=y(i)=z(i)$.

\begin{lemma}\label{specialsubfamily}
Given an arbitrary $\ma{F}:=SHF(4;n,q,\{2,2\})$, we can obtain a subfamily $\ma{F}^*\s\ma{F}$ which contains no special columns with respect to $\ma{F}^*$ by deleting at most $8q$ columns of $\ma{F}$.
\end{lemma}

\begin{proof}
  We will use a greedy algorithm to construct $\ma{F}^*$. Delete $x_1$ from $\ma{F}$ if $x_1$ has a special coordinate in $\ma{F}$. Denote $\ma{F}_1=\ma{F}-\{x_1\}$. In general, if $x_{i+1}\in \ma{F}_i$ has a special coordinate in $\ma{F}_i$, we delete $x_{i+1}$ from $\ma{F}_i$ and then denote $\ma{F}_{i+1}=\ma{F}_i-\{x_{i+1}\}$. Continue the deleting procedure until we get an $\ma{F}^{*}$ with no columns containing a special coordinate in it. We say that some symbol $a\in\{1,\ldots,q\}$ in row $i$ is {\it responsible} for some column $x$ if $x$ is deleted and it contains $x(i)=a$ as a special coordinate when it is deleted. We claim that in order to obtain $\ma{F}^*$ at most $8q$ columns will be deleted from $\ma{F}$.

  To see this, just notice that for $1\le i\le 4$, any symbol $a\in\{1,\ldots,q\}$ in row $i$ can be responsible for at most two columns, since otherwise let $x,y,z$ be the three columns for which $a$ is responsible (assume we delete $x$ first), then by definition it holds that $x(i)=y(i)=z(i)=a$ and hence $a$ is not a special component of $x$ (note that $y,z$ have not yet been deleted) and it cannot be responsible for $x$! On the other hand, any deleted column must contain at least one symbol which is responsible for it. Therefore, a simple double-counting argument yields that we have deleted at most $8q$ columns. \end{proof}

  \begin{remark}\label{property}
    The following property is straightforward: If $\ma{F}^*$ is an $SHF(4;n,q,\{2,2\})$ which contains no special columns, then for any $1\le i\le 4$ and any $x\in\ma{F}$, there exist at least two columns $y,z\in\ma{F}\setminus\{x\}$ satisfying $x(i)=y(i)=z(i)$.
  \end{remark}

\begin{lemma}\label{linearsubfamily}
  Given an arbitrary $\ma{F}:=SHF(4;n,q,\{2,2\})$, we can obtain a linear subfamily by deleting at most $8q$ columns of it.
\end{lemma}

\begin{proof}
  By Lemma \ref{specialsubfamily} one can deduce that there exists an $\ma{F}^*\s\ma{F}$ with size at least $n-8q$ which contains no special columns with respect to $\ma{F}^*$. We claim that $\ma{F}^*$ must be linear. Assume, to the contrary, that there exist two distinct columns $x,y\in\ma{F}^*$ which contain at least two identical components. Without loss of generality, assume that $x(1)=y(1)$ and $x(2)=y(2)$. Pick a column $z\in\ma{F}^*\setminus\{x,y\}$ such that $x(3)=z(3)$. Note that the existence of such a $z$ is guaranteed by Remark \ref{property}. Now let us consider $y(4)$. On one hand, if there exists some column $w\in\ma{F}^*\setminus\{x,y,z\}$ such that $w(4)=y(4)$, then obviously no row of $\ma{F}^*$ can separate $\{x,w\}$ and $\{y,z\}$ (see Table \ref{22} for an illustration of the proof), a contradiction. On the other hand, if such a $w$ does not exist, then by Remark \ref{property} one can infer that $x(4)=y(4)=z(4)$, which implies that no row can separate $\{x\}$ and $\{y,z\}$ (see Table \ref{12} for an illustration of the proof), a contradiction, too.

  \begin{table}[ht]
  \begin{center}
    \begin{tabular}{|c|c|c|c|c|}
      \hline
       & $x$ & $y$ & $z$ & $w$ \\\hline
      $V_1$ & $x(1)$ & $y(1)$ & &\\\hline
      $V_2$ &  $x(2)$ & $y(2)$ && \\\hline
      $V_3$ & $x(3)$ &  & $z(3)$ & \\\hline
      $V_4$ &  & $y(4)$ &  & $w(4)$\\\hline
    \end{tabular}
  \end{center}
   \caption{No row can separate $\{x,w\}$ and $\{y,z\}$}\label{22}
    \end{table}



      \begin{table}[ht]
  \begin{center}
    \begin{tabular}{|c|c|c|c|}
      \hline
       & $x$ & $y$ & $z$  \\\hline
      $V_1$ & $x(1)$ & $y(1)$ & \\\hline
      $V_2$ &  $x(2)$ & $y(2)$ &\\\hline
      $V_3$ & $x(3)$ &  & $z(3)$ \\\hline
      $V_4$ & $x(4)$ & $y(4)$ & $z(4)$  \\\hline
    \end{tabular}
  \end{center}
   \caption{No row can separate $\{x\}$ and $\{y,z\}$}\label{12}
    \end{table}
\end{proof}

\begin{theorem}\label{shf22}
    Let $\ma{F}$ be an $SHF(4;n,q,\{2,2\})$. Then for sufficiently large $q$, it holds that $n=o(q^2)$.
\end{theorem}

\begin{proof}
  By Lemma \ref{linearsubfamily} one can deduce that there exists a linear subfamily $\ma{F}^*\s\ma{F}$ such that $|\ma{F}^*|\ge n-8q$. Corollary \ref{linearshf} implies that $n-8q=o(q^2)$ and hence the theorem follows trivially.
\end{proof}

\begin{remark}
  It is worth mentioning that in \cite{ippe} it was shown $C(4,q,\{2,2\}+\{1,1,1\})=o(q^2)$, where $\{2,2\}+\{1,1,1\}$ means that a family is simultaneously $\{2,2\}$-separating and $\{1,1,1\}$-separating. Theorem \ref{shf22} is stronger than their result in the sense that we show $\{2,2\}$-separating property already guarantees the $o(q^2)$ magnitude.
\end{remark}

The following lemma is a simple consequence of Lemma \ref{johnson} by taking $N=\sum_{i=1}^t w_i$ and $l=1$.

\begin{lemma} \label{johnsonl=1}
   Let $u=\sum_{i=1}^t w_i$, then it holds that $$C(u,q,\{w_1,...,w_t\})\le q+(u-1)+C(u-1,q,\{w_1-1,...,w_t\}).$$ In fact, in the right hand side of the inequality we can choose the minus of 1 to be after an arbitrary $w_i,~1\le i\le t$.
\end{lemma}

Now we are able to prove the main theorem of this section.


\begin{proof}[\textbf{Proof of Theorem \ref{upperbound}}]
  Let $u=\sum_{i=1}^t w_i$. For $t\ge 3$, by applying Lemma \ref{johnsonl=1} repeatedly for $u-3$ times one can infer
  $$C(u,q,\{w_1,...,w_t\})\le (u-3)q+(u-3)(u-1)+C(3,q,\{1,1,1\}),$$ which implies that $C(u,q,\{w_1,...,w_t\})=o(q^2)$ since by (\ref{111}) we have $C(3,q,\{1,1,1\})=o(q^2)$.

  For $t=2$ and $\min\{w_1,w_2\}\ge2$, by applying Lemma \ref{johnsonl=1} repeatedly for $u-4$ times one can infer
  $$C(u,q,\{w_1,...,w_t\})\le (u-4)q+(u-4)(u-1)+C(4,q,\{2,2\}),$$ which implies that $C(u,q,\{w_1,...,w_t\})=o(q^2)$ since by Theorem \ref{shf22} we have $C(4,q,\{2,2\})=o(q^2)$.
\end{proof}

One may wonder that whether the upper bound obtained by Theorem \ref{upperbound} is tight. The following result collects all the known tight cases of Theorem \ref{upperbound}.

\begin{proposition}\label{lowerbd}
For sufficiently large $q$, it holds that $$q^{2-o(1)}<C(3,q,\{1,1,1\})=o(q^2)$$ and $$q^{2-o(1)}<C(4,q,\{1,1,1,1\})\le C(4,q,\{1,1,2\})\le C(4,q,\{2,2\})=o(q^2).$$
\end{proposition}

\begin{proof}
   Observe that $\{1,1,1,1\}$-separating implies $\{1,1,2\}$-separating and $\{1,1,2\}$-separating implies $\{2,2\}$-separating. Thus the proposition is a simple consequence of (\ref{111}) and (\ref{1111}).
\end{proof}

%

\section{An upper bound for $SHF(N;n,t,\{w_1,\ldots,w_t\})$}

In this section, we will present an upper bound for $SHF(N;n,t,\{w_1,w_2,\ldots,w_t\})$ (note that here $q=t$).
In the proof we will use the asymptotic $\binom{n}{w}\thickapprox\fr{n^w}{w!}$, where $w$ is the largest value among the $w_i$'s.
Thus our bound is valid for sufficiently large $n$ (compared with the $w_i$'s).
The main result can be stated as the following theorem.

\begin{theorem}\label{tw}
Given a positive integer $t$, let $w_1,\ldots,w_t$ be $t$ fixed integers such that $\min\{w_i\mid 1\le i\le t\}\ge 2$. Denote $u=\sum_{i=1}^t w_i$.
Let $\ma{F}$ be an $SHF(N;n,t,\{w_1,\ldots,w_t\})$ such that $n$ is sufficiently large compared with the $w_i$'s.
Then the following two statements hold.
\begin{itemize}
 \item [(a)] Denote $p^*=\max_{\sum_{i=1}^t p_i=1,~0\le p_i\le 1}\sum_{\pi\in S_t}\prod_{i=1}^tp_{\pi(i)}^{w_i-1}$ and $g(q,j+1)=\fr{q(q-1)\cdots(q-j)}{q^{j+1}}$, where $S_t$ denotes the symmetric group 
 defined on a finite set of $t$ elements. Then it holds that
 \begin{equation*}
   \begin{aligned}
     C(N,t,\{w_1,\ldots,w_t\})&\le C(p^*N,t,\{1,\ldots,1\})+u-t\\
     &\le \min_{0\le j\le t-2}(t-j-1)(\fr{q-j}{t-j-1})^{g(q,j+1)p^*N}+u-t.
   \end{aligned}
 \end{equation*}
 \item [(b)]For $w_1=\cdots=w_t=w\ge2$, it holds that
 \begin{equation*}
   \begin{aligned}
     C(N,t,\{w,\ldots,w\})&\le C(t!(\fr{1}{t})^{t(w-1)}N,t,\{1,\ldots,1\})+u-t\\
     &\le\min\{2^{\fr{(t!)^2N}{t^{tw-1}}},(t-1)(\fr{t}{t-1})^{\fr{t!N}{t^{tw-t}}}\}+u-t.
   \end{aligned}
 \end{equation*}
\end{itemize}
\end{theorem}

\begin{proof}
Without loss of generality, set the alphabet set of $\ma{F}$ to be $[t]:=\{1,\ldots,t\}$.
Denote by $M$ the representation matrix of $\ma{F}$. For each $1\le i\le t$, let $p_i$ be the fraction of the symbol $i$ in $M$.
We can also view $p_i$ as the probability of a randomly chosen entry of $M$ being equal to $i$. It is easy to see that $\sum_{i=1}^t p_i=1$ and $0\le p_i\le 1$.
We pick randomly and independently $t$ disjoint subsets $C_1^{'},\ldots,C_t^{'}\s\ma{F}$ such that $|C_i^{'}|=w_i-1$ for $1\leq i\leq t$.
One can compute that the probability that a row $f$ of $M$ separates $C_1^{'},\ldots,C_t^{'}$ is at most
\begin{equation}\label{pr1}
  \begin{aligned}
    f(p_1,\ldots,p_t):=\fr{\sum_{\pi\in S_t}\prod_{i=1}^{t}\binom{p_{\pi(i)}n}{w_{i}-1}}{\binom{n}{w_1-1}\cdot\binom{n-w_1+1}{w_2-1}\cdots\binom{n-w_1-\cdots-w_{t-1}+t-1}{w_t-1}}.
  \end{aligned}
\end{equation}
\noindent Denote
\begin{equation}
  \begin{aligned}
    p^*:=\max_{\sum_{i=1}^t p_i=1,~0\le p_i\le 1}f(p_1,\ldots,p_t).
  \end{aligned}
\end{equation}
By linearity of expectation, there exist $t$ column sets $C_1^{'},\ldots,C_t^{'}\s\ma{F}$ which are separated by at most $p^*N$ rows of $M$.
Let $T$ be the collection of rows that separate $C_1^{'},\ldots,C_t^{'}\s\ma{F}$ such that $|T|\le p^*N$.
For $t$ distinct columns $c_1,\ldots,c_t\in\ma{F}\setminus(C_1^{'}\cup C_2^{'}\cdots\cup C_t^{'})$, there must exist at least one row $f\in T$ that separates $c_1,\ldots,c_t$,
since otherwise no row of $M$ will separate $C_1^{'}\cup\{c_1\},\ldots,C_t^{'}\cup\{c_t\}$, contradicting the $\{w_1,\ldots,w_t\}$-separating property of $M$.
Therefore, one can conclude that the submatrix formed by rows $T$ and columns $\ma{F}\setminus(C_1^{'}\cup C_2^{'}\cdots\cup C_t^{'})$ must be a representation matrix of an $SHF(|T|;n-(u-t),t,\{1,\ldots,1\})$ (there are $t$ 1's in total). Thus one can infer that
\begin{equation*}
  \begin{aligned}
    n-u+t\le C(|T|,t,\{1,\ldots,1\}),
  \end{aligned}
\end{equation*}
which implies that
\begin{equation}\label{fp}
  \begin{aligned}
    n\le C(|T|,t,\{1,\ldots,1\})+u-t\le C(p^*N,t,\{1,\ldots,1\})+u-t.
  \end{aligned}
\end{equation}
We can express each term in the summation of (\ref{pr1}) as follows
\begin{equation}\label{pr2}
  \begin{aligned}
   \fr{\binom{p_{\pi(1)}n}{w_1-1}}{\binom{n}{w_1-1}}\cdot\fr{\binom{p_{\pi(2)}n}{w_2-1}}{\binom{n-w_1+1}{w_2-1}}\cdots
   \fr{\binom{p_{\pi(t)}n}{w_t-1}}{\binom{n-w_1-\cdots-w_{t-1}+t-1}{w_t-1}}.
  \end{aligned}
\end{equation}
For sufficiently large $n$ one can infer that (\ref{pr2}) attains the maximal only if $\min\{p_i\mid 1\le i\le t\}>a>0$ holds for some constant $a>0$.
Thus we can set $\binom{p_{\pi(i)}n}{w_i-1}\thickapprox\fr{(p_{\pi(i)}n)^{w_i-1}}{(w_i-1)!}$ for each $1\le i\le t$.
One can compute that $f(p_1,\ldots,p_t)$ approximates
\begin{equation}\label{pr3}
  \begin{aligned}
    \sum_{\pi\in S_t}\prod_{i=1}^tp_{\pi(i)}^{w_i-1},
  \end{aligned}
\end{equation}
which together with (\ref{phfbound}) and (\ref{fp}) complete the first part of the theorem.
To prove the second part of the theorem 
for the special case $w_1=\cdots=w_t=w$, by taking the logarithm of $\prod_{i=1}^{t}\binom{p_{\pi(i)}n}{w-1}$ and using the concave property of $\log(\cdot)$ one can infer that $f(p_1,\ldots,p_t)$ attains its maximality $p^*$ if and only if $p_1=\cdots=p_t=\fr{1}{t}$.
For sufficiently large $n$, using (\ref{pr3}) one can infer that $p^*$ approximates
\begin{equation}
  \begin{aligned}
    t!(\fr{1}{t})^{t(w-1)},
  \end{aligned}
\end{equation}
which together with (\ref{phfbound2}) complete the proof of the second part of the theorem.
\end{proof}

\section{A tight upper bound for $SHF(N;2,q,\{1,w\})$}

In this section we will provide a lower bound for the minimal $N$ such that there exists an $SHF(N;n,2,\{1,w\})$ with $n>N$, which can be stated as the following theorem.
\begin{theorem}\label{shangguanbinary}
    For all $w\geq 3$ and for all $N<\fr{15+\sqrt{33}}{24} (w-2)^2$, it holds that $C(N,q,\{1,w\})\le N$. Or equivalently, $N(w)\ge\fr{15+\sqrt{33}}{24} (w-2)^2$.
\end{theorem}

\begin{remark}
  Indeed we have a tight bound $C(N,q,\{1,w\})=N$ for $N<\fr{15+\sqrt{33}}{24} (w-2)^2$, since an $N\times N$ identity matrix satisfies the $\{1,w\}$-separating property.
\end{remark}
We will need a few lemmas to prove Theorem \ref{shangguanbinary}. 

%

\begin{lemma}\label{cffandfpc}
    Every $CFF(N;n,w)$ is also an $SHF(N;n,2,\{1,w\})$ and every $SHF(N;n,2,\{1,w\})$ induces a $CFF(2N;n,w)$.
\end{lemma}

\begin{proof}
    Denote by $M$ and $M^*$ the representation matrices of an $SHF(N;n,2,\{1,w\})$ and a $CFF(N;n,w)$ respectively. Given $M^*$, by the $w$-cover-free property it holds that for each column and arbitrary $w$ other columns there exists a row in which the first column is 1 and the remaining $w$ columns are all 0. If we view the columns of $M^*$ as columns of some binary separating hash families, then $M^*$ apparently satisfies the $\{1,w\}$-separating property, which implies that $M^*$ also represents an $SHF(N;n,2,\{1,w\})$.

    On the other hand, given $M$, by replacing the 0 entry in $M$ by 10 and the 1 entry by 01, we obtain a $2N\times n$ matrix which is denoted by $M_1$.
    It suffices to verify that $M_1$ is a representation matrix of a $CFF(2N;n,w)$.
    For each column and arbitrary $w$ other columns of $M_1$, let us consider the corresponding columns of $M$.
    By the $\{1,w\}$-separating property, for these $w+1$ columns there is a row of $M$ having the configuration $10\cdots 0$ or $01\cdots 1$, which is translated to
    $$\left(\begin{array}{cccc}
        0 & 1 & \cdots & 1 \\
        1 & 0 & \cdots & 0
      \end{array}\right)~or~
      \left(\begin{array}{cccc}
        1 & 0 & \cdots & 0 \\
        0 & 1 & \cdots & 1
      \end{array}\right),
    $$
    \noindent in $M_1$. Note that the second row of the first submatrix and the first row of the second submatrix satisfy the $w$-cover-free property.
    If we view $M_1$ as a representation matrix for some $\ma{F}\s2^X$, where $|X|=2N$, then by the above discussions we can conclude that $\ma{F}$ is $w$-cover-free.
\end{proof}

\begin{lemma}\label{fpctightbd}
   Denote $N^*(w)$ the minimal $N$ such that there exists a $CFF(N;n,w)$ with $n>N$. And denote $N(w)$ the minimal $N$ such that there exists an $SHF(N;n,2,\{1,w\})$ with $n>N$. Then for $w\ge3$, it holds that $N^*(w-2)\le N(w)\le N^*(w)$.
\end{lemma}

\begin{proof}
Denote by $M$ the representation matrix of an $SHF(N;n,2,\{1,w\})$ with $N=N(w)$, then we have $n>N$ by the definition of $N(w)$. First of all, the upper bound in the inequality follows from the fact that every $CFF(N;n,w)$ is also an $SHF(N;n,2,\{1,w\})$, which is shown in Lemma \ref{cffandfpc}. It remains to prove the lower bound. Replace the 0 entry in $M$ by 10 and the 1 entry by 01. We obtain a $2N\times n$ matrix with constant column weight $N$. Denote this new matrix by $M_1$. By Lemma \ref{cffandfpc}, $M_1$ is the representation matrix of a $CFF(2N;n,w)$.

By Lemma \ref{cffpro}, deleting from $M_1$ an arbitrary column and the rows containing a 1 in it leads to a new matrix $M_2$, which is the representation matrix of a $CFF(N;n-1,w-1)$.

We claim that there must exist a column in $M_2$ of weight at least two. Denote by $c$ the column deleted from $M_1$. If some column $c'\in M_2$ is of weight one, then one can verify that $c$ and $c'$ have exactly $N-1$ identical components in $M$. If $M_2$ contains two columns of weight 1, then in $M$ there are two distinct columns that have exactly $N-1$ identical components with $c$. Then it is not hard to show that no row of $M$ can separate $c$ and these two columns.
Therefore, $M_2$ contains at most one column of weight 1. The claim follows from the simple fact that $n-1\ge N=N(w)\ge N(3)>6$, where the last inequality follows from (\ref{N(w)shangguan}).

Take an arbitrary column of $M_2$ with weight at least two. Delete from $M_2$ this column and the rows containing a 1 in it. Again, by Lemma \ref{cffpro}, the new matrix is the representation matrix of a $CFF(N';n-2,w-2)$ satisfying $N'\le N-2<n-2$ since we have assumed that $n>N$. Thus one can deduce that $N'\ge N^*(w-2)$ and hence the lower bound $N(w)\ge N^*(w-2)$ follows immediately.
\end{proof}

%

\begin{proof}[\textbf{Proof of Theorem \ref{shangguanbinary}}]
    Theorem \ref{shangguanbinary} is a direct consequence of Lemmas \ref{grouptesting} and \ref{fpctightbd}.
\end{proof}

\section{Concluding remarks}

This paper provides several new upper bounds for separating hash families. The following remaining open problems seem to be interesting.

Theorem \ref{lowerbd} implies that Theorem \ref{upperbound} is tight for several small parameters. For the general situation, the following conjecture seems reasonable.
\begin{conjecture}\label{conjecture1}
  Let $w_1,\ldots,w_t$ be fixed integers such that $\{w_1,\ldots,w_t\}\neq\{1,w\}$ for any $w$. Then for sufficiently large $q$ we have $C(\sum_{i=1}^t w_i,q,\{w_1,\ldots,w_t\})>q^{2-o(1)}$.
\end{conjecture}
Since $u$-perfect hash family satisfies $\{w_1,\ldots,w_t\}$-separating property for arbitrary $\sum_{i=1}^t w_i=u$,
to verify Conjecture \ref{conjecture1}, it suffices to show that $C(t,q,\{1,\ldots,1\})>q^{2-o(1)}$ (there are $t$ 1's in total) for sufficiently large $q$ and every fixed integer $t$.
Lemma 6.1 of \cite{shf} shows that $t$-uniform $t$-partite linear hypergraphs which contain no rainbow cycles of length $k$ for any $3\le k\le t$ are good candidates for $t$-perfect hash families.
Thus to prove Conjecture \ref{conjecture1} it suffices to construct sufficiently large hypergraphs containing no rainbow cycles.
Unfortunately, this is not a simple question. Some constructions of this type using additive number theory can be found in \cite{ge2017sparse}.

For the upper bound of $SHF(N,n,q,\{w_1,\ldots,w_t\})$ under the situation $(u-1)\nmid N$,
it may be a difficult problem to determine whether $n=\Theta(q^{\lc\fr{N}{u-1}\rc})$ or $n=o(q^{\lc\fr{N}{u-1}\rc})$ for general $w_1,\ldots,w_t$.
We have a conjecture for the case $w_1=\cdots=w_t=1$, which is stated as follows.
\begin{conjecture}\label{conj2}
  Let $t\ge 3$ be a positive integer. For $w_1=\cdots=w_t=1$ and $(t-1)\nmid N$, it holds that $C(N,q,\{1,\ldots,1\})=o(q^{\lc\fr{N}{t-1}\rc})$.
\end{conjecture}
\noindent Note that Theorem \ref{upperbound} implies that Conjecture \ref{conj2} is true for $N=t\ge 3$.
One possible strategy to attack this conjecture is to use the hypergraph removal lemma (see for example, \cite{graphremovallemmas}), which is a generalization of the graph removal lemma.
The authors have attempted to use this tool to prove $C(5,q,\{1,1,1\})=o(q^3)$ for sufficiently large $q$.
Unfortunately, we failed to accomplish this since the structure of the hypergraphs needed to be considered is too complicated
(indeed to prove Theorem \ref{upperbound} it suffices to consider linear hypergraphs).

We also have a conjecture for the approximate magnitude of $N(w)$.
\begin{conjecture}\label{con3}
  $\lim_{N\rightarrow\infty}\fr{N(w)}{w^2}=1$.
\end{conjecture}
\noindent Note that Lemma \ref{fpctightbd} implies that this conjecture is equivalent to the weaker form of Conjecture \ref{cff}.
\bibliographystyle{plain}
\bibliography{shf}

\end{document}